\documentclass[a4 paper,11pt]{article}
\usepackage[active]{srcltx}

\usepackage{epsfig}
\usepackage{amsmath}
\usepackage[latin1]{inputenc}
\usepackage{amsmath,amsthm,amsfonts,amssymb,amscd}
\usepackage{graphicx}

\theoremstyle{plain}
\newtheorem{theorem}{Theorem}[section]
\newtheorem{proposition}[theorem]{Proposition}
\newtheorem{lemma}[theorem]{Lemma}
\newtheorem{corollary}[theorem]{Corollary}

\theoremstyle{definition} \theoremstyle{remark}
\newtheorem{remark}[theorem]{Remark}

\newtheorem{definition}[theorem]{Definition}

    \newtheoremstyle{TheoremNum}
        {\topsep}{\topsep}              
        {\itshape}                      
        {}                              
        {\bfseries}                     
        {.}                             
        { }                             
        {\thmname{#1}\thmnote{ \bfseries #3}}
    \theoremstyle{TheoremNum}

\begin{document}

\begin{center}
\textbf{\huge Curvature and partial hyperbolicity}

\vspace{10 mm}  

{\large FERNANDO CARNEIRO\footnote{First author partially supported by CAPES} and ENRIQUE R. PUJALS\\}
{\large IM-UFRJ, Av. Athos da Silveira Ramos 149, Centro de Tecnologia - Bloco C Cidade Universit\'aria - Ilha do Fund\~ao. CEP 21941-909 Rio de Janeiro - RJ - Brazil \\(e-mail: fernando@impa.br)\\
IMPA, Estrada Dona Castorina, 110, CEP 22460-320, Rio de Janeiro, Brazil \\(email: enrique@impa.br)}

\end{center}

\begin{abstract}
Using quadratic forms, we stablish a criteria to relate the curvature of a Riemannian manifold and partial hyperbolicity of its geodesic flow. We show some examples which satisfy the criteria and another which does not satisfy it but still has a partially hyperbolic geodesic flow.
\end{abstract}

\large

\section{Introduction}

It is well know fact that hyperbolic dynamics has been one of the most  sucessfull theory in dynamical systems. But soon was realized that there  is an easy  way to relax hyperbolicity, called partial hyperbolicity, which allows the tangent bundle to split into invariant subbundles $TM=E^s\oplus E^c\oplus E^u,$  such that the behavior of vectors in $E^s, E^u$ is similar to the hyperbolic case, but vectors in $E^c$ may be neutral for the action of the tangent map. This notion arose in a natural way   in the context of time one maps of Anosov flows, frame flows, group extensions and it was possible to show the existence of open set od partially hyperbolic diffeomorphisms which are not hyperbolic. See \cite{BP}, \cite{Sh}, \cite{M}, \cite{BD}, \cite{BV} for examples of these systems  and \cite{HP}, \cite{PS} for an overview.

However, until the results in \cite{CP}, partially hyperbolic systems where unknown in the context of geodesic flows induced by Riemannian metrics. In fact, in \cite{CP} it was proved that
{\em for some compact locally symmetric space $(M,g)$ whose sectional curvature takes values in the whole interval $[-4a^2,-a^2]$, there is a metric $g^*$ in $M$ such that its geodesic flow is partially hyperbolic but not Anosov.}

On the other hand, from the works of J. Lewowicz, qudratic forms have been a powerfull tool to characterize expansive dynamics in general and hyperbolic ones in particular (see \cite{L1, L4}); moreover, this approach  have been extended to the context of geodesic flows (see for instance \cite{L}, \cite{P},  \cite{R1} and \cite{R2}) and billiards (see \cite{Mar}, \cite{MPS}).  Those techniques, using previous results by Potapov, have been extended and generalized in \cite{W2}.  

In the present paper, we explore the use of quadratic forms  for the particular context of partially hyperbolic geodesic flows. Moreover, we revisite the examples of partially hyperbolic geodesic flows provided in \cite{CP} using quadratic forms and we relate the partially hyperbolicity with the curvature tensor.  

In the second section, we give some definitions and the criteria we are going to apply to get partially hyperbolic examples.

In the thid section, we prove the main theorem with the help of the criteria and to get a corollary, which is going to be useful to prove that some examples are partially hyperbolic.

In the fourth and last section we present some examples of geodesic flows which are going to be partially hyperbolic. 

\section{Definitions}

Before given the results, we need to introduce a few definitions.

A partially hyperbolic flow $\phi_t : M \to M$ in the manifold $M$ generated by the vector field $X: M \to TM$ is a flow such that its quotient bundle $TM / \langle X \rangle$ (assuming that $X$ has not singularities) have an invariant splitting $TM / \langle X \rangle = E^s \oplus E^c \oplus E^u$ such that these subbundles are non trivial and with the following properties:

$$ d\phi_t(x) (E^s(x)) = E^s(\phi_t(x)), d\phi_t(x) (E^c(x)) = E^c(\phi_t(x)), d\phi_t(x) (E^u(x)) = E^u(\phi_t(x)), $$

$$ || d\phi_t(x)|_{E^s} || \leq C \exp(t \lambda), || d\phi_{-t}(x)|_{E^u} || \leq C \exp(t \lambda), $$ $$C \exp(t \mu) \leq || d\phi_t(x)|_{E^c} || \leq C \exp(- t \mu),$$

\noindent for $\lambda < \mu < 0 < C$.

Let $(M,g)$ be a Riemannian manifold, $p: TM \to M$ its tangent bundle, $\phi_t: SM \to SM$ be its geodesic flow. The geodesic flow is always a Reeb flow \cite{P}, i.e., given a $2n+1$-dimensional manifold $N$, an one form $\tau$ such that $\tau \wedge d\tau^n$ is a volume form, the Reeb vector field $Y$ is the vector field such that $i_Y \tau = 1$ and $i_Y d\tau = 0$, and its flow is a Reeb flow. The kernel of $\tau$ is called the contact structure of the contact manifold $(N,\tau)$. It is allways invariant under the flow and transversal to the Reeb vector field \cite{P}. 

The double tangent bundle $TTM$ is isomorphic to the vector bundle $\mathcal{E} \to TM$, $\mathcal{E} = \pi^*TM \oplus \pi^*TM$, with fiber $\mathcal{E}_v = T_{\pi(v)}M \oplus T_{\pi(v)}M$. We define the isomorphism as $$\mathcal{I}: TTM \to \mathcal{E} : Z \to ((\pi \circ V)'(0), \frac{DV}{dt}(0)),$$ where $V$ is a curve on $TM$ such that $V'(0) = Z$. The Sasaki metric in the double tangent bundle is the pull-back of the metric $\widetilde{g}$ on $\mathcal{E}$: $\widetilde{g}_v((\eta_1,\varsigma_1),(\eta_2,\varsigma_2)) = g_{\pi(v)}(\eta_1,\eta_2) + g_{\pi(v)}(\varsigma_1,\varsigma_2)$ where $(\eta_1,\varsigma_1),(\eta_2,\varsigma_2) \in T_{\pi(v)}M \oplus T_{\pi(v)}M$. The contact structure $\xi(SM)$ of the geodesic flow is identified with the vector bundle $\mathcal{E'} \to SM$, $\mathcal{E'}_v \subset \mathcal{E}_v$, for all $v \in SM$, whose fiber at $v$ is $v^{\bot} \oplus v^{\bot}$, where $v^{\bot} = \{ w \in T_{\pi(v)}M : g(w,v) = 0 \}$. The derivative of the geodesic flow, given the identification of $TTM$ and $\mathcal{E}$ is $d_v\phi_t(\eta,\varsigma) = (J(t),J'(t))$, $J(t) \in T_{\phi_t(v)}M$ such that $J''(t) + R(\phi_t(v),J(t))\phi_t(v) = 0$ \cite{Ba1},\cite{P}.

Let $\pi: \xi(SM) \to SM$ be the contact structure of the geodesic flow of $(M,g)$. Let $Q: \xi(SM) \to \mathbb{R}$ be a nondegenerate quadractic form of constant signature $(l,m)$. Let $\mathcal{C}_+(x,v) = \{ \eta \in \xi(x,v) : Q_{(x,v)}(\eta) > 0 \}$ be its positive cone, $\mathcal{C}_-(x,v) = \{ \eta \in \xi(x,v): Q_{(x,v)}(\eta) < 0 \}$ be its negative cone and $\mathcal{C}_0(x,v) = \{ \eta \in \xi(x,v): Q_{(x,v)}(\eta) = 0 \}$ be their boundary. The criteria says the following:

\begin{lemma}
If $$\frac{d}{dt} \mathcal{Q}(\eta,\varsigma) > 0$$ for all $(\eta,\varsigma) \in \mathcal{C}_0(x,v)$, $(x,v) \in SM$, then the flow $\phi_t$ is strictly $Q$-separated. This criteria and reversibility of the geodesic flow imply it has a partially hyperbolic splitting $\xi = E^s \oplus E^c \oplus E^u$, $dim(E^{\sigma}) = l$, $\sigma=s,u$.
\end{lemma}

\begin{proof}
See the proof in \cite{W1}.
\end{proof}

\section{Main results}

In this section we state the main theorem and a corollary which is going to be useful to prove that some geodesic flows are partially hyperbolicity.

Let $(M^n,g)$ be a $n$-dimensional Riemannian manifold, $\nabla$ its Levi-Civita connection. Let $R_x: T_xM \times T_xM \times T_xM \to T_xM$ be its curvature tensor. Let $v \in T_xM$, then $R_x(v,\cdot)v: T_xM \to T_xM$ is a symmetric linear operator. We can restrict it to $v^\bot$, since $R(v,v)v = 0$. So $R(v,\cdot)v : v^\bot \to v^\bot$ is a symmetric linear operator. So we can diagonalize it: there are eigenvalues $\lambda_1 \geq \lambda_2 \geq \ldots \geq \lambda_{n-1}$ and eigenvectors $v_1,v_2,\ldots,v_{n-1}$ such that $R(v,v_k)v = \lambda_k v_k$.

Suppose there is an $1 < r < n-2$ such that $\lambda_{r}(v) > \lambda_{r+1}(v)$ for each $v \in TM$. Then we are able to define $A(v) = \mathbb{R}v_1 \oplus \ldots \oplus \mathbb{R}v_r$ and $B(v) = \mathbb{R}v_{r+1} \oplus \ldots \oplus \mathbb{R}v_{n-1}$. It is easy to see that $A(v) \oplus B(v) = v^\bot$. Let $Gr(r,TM)$ be the Grassmanian bundle of $r$-dimensional subbundles of $TM$. Then $A,B : TM \to Gr(r,TM)$. Also $A(cv) = A(v)$, $B(cv) = B(v)$ for all $c \in \mathbb{R}$, $c \neq 0$. So we consider $A,B : SM \to Gr(r,TM)$, where $SM$ is the unitary tangent bundle of $M$.

Let $P_{A(v)} : T_{p(v)}M \to A(v)$ be the orthogonal projection to $A(v)$. Let $A' = P_{A(v)}' = \frac{d}{dt}|_{t=0} P_{A(\phi_t(v))}$. Let $\eta_A$ be $P_A \eta$. Let $K_A$ be the restriction $R(v,\cdot)v : A(v) \to A(v)$ and $K_B$ be the restriction $R(v,\cdot)v : B(v) \to B(v)$.

Linearization of the derivative of the geodesic flow gives you the system of equations $$\eta' = \varsigma , \varsigma' = - R(v,\eta)v,$$ for $(\eta,\varsigma) \in T_xM \oplus T_xM \cong T_vTM$. 

\begin{theorem}
\label{t.maintheo}
Let $(M,g)$ be a Riemannian manifold. Suppose $A: SM \to Gr(r,TM): v \to A(v) \subset T_{p(v)}M$ is a continuous function (smooth along geodesics). Let the quadractic form be $\mathcal{Q}^c(\eta,\varsigma) = g(\eta_A,\varsigma_A) - c^2g(\eta_B,\eta_B) -  g(\varsigma_B,\varsigma_B)$, where $c$ is a positive real number. If 
$\frac{d}{dt} \mathcal{Q}^c(\eta,\varsigma) = \widetilde{g}(S^c(\eta,\varsigma),(\eta,\varsigma))$ is positive for the following matrix $$S^c = \begin{bmatrix} -K_A & 0 & c^2A' & \frac{1}{2} A' \\ 0 & Id & \frac{1}{2} A' & A' \\ c^2A' & \frac{1}{2} A' & 0 & -c^2Id + K_B \\ \frac{1}{2} A' & A' & -c^2Id + K_B & 0 \end{bmatrix}$$ and for all $(\eta,\varsigma) \in \mathcal{C}_+(x,v)$, $(x,v) \in SM$, then the geodesic flow of the Riemannian manifold $(M,g)$ is partially hyperbolic $dim(E^{\sigma}) = r$, $\sigma = s,u$.
\end{theorem}


\begin{proof}
We have to calculate the following derivative:
\begin{eqnarray*}
\frac{d}{dt} \mathcal{Q}^c(\eta,\varsigma) & = & \frac{d}{dt} (g(\eta_A,\varsigma_A) - c^2 g(\eta_B,\eta_B) - g(\varsigma_B,\varsigma_B)) \\ & = & g(\varsigma_A,\varsigma_A) - g(R(v,\eta)v,\eta_A) + g(\eta_{A'},\varsigma_A) + g(\eta_A,\varsigma_{A'}) \\ & - & 2c^2  g(\eta_B,\varsigma_B) + 2 g(R(v,\eta)v,\varsigma_B) - 2c^2 g(\eta_{B'},\eta_B) - 2 g(\varsigma_{B'},\varsigma_B).
\end{eqnarray*}
$P_A$ is the orthogonal projection to $A$, $P_B$ is the orthogonal projection to $B$, then $P_A (P_A)' = (P_A)' P_B$ and $P_B (P_A)' = (P_A)' P_A$. It implies that
\begin{eqnarray*}
\frac{d}{dt} \mathcal{Q}^c(\eta,\varsigma) & = & g(\varsigma_A,\varsigma_A) - g(R(v,\eta)v,\eta_A) + g((P_A)' \eta_B,\varsigma_A) + g(\eta_A,(P_A)' \varsigma_B) \\ & - & 2 c^2 g(\eta_B,\varsigma_B) + 2 g(R(v,\eta)v,\varsigma_B) + 2 c^2 g((P_A)' \eta_{A},\eta_B) + 2 g((P_A)' \varsigma_{A},\varsigma_B)). 
\end{eqnarray*}
\end{proof}

\noindent Now suppose we are able to define two functions, $\alpha,\beta: SM \to \mathbb{R}_+$ such that 
\begin{itemize}
\item[i.] $-\alpha(v)^2 > max \{\lambda_i\}_{i=1}^r$, 
\item[ii.] $-\alpha(v)^2 < - \beta(v)^2 < \lambda_i$ if $i=r+1,\ldots,n-1$,
\item[iii.] there is a constant $e \in \mathbb{R}_+$ such that $\beta(v) < e < \alpha(v)$ for all $v \in SM$,
\end{itemize}

\noindent then we are able to fix $c \in \mathbb{R}_+$ such that $c := e$.

\begin{corollary}
\label{c.maintheo}
Under the hypothesis of theorem \ref{t.maintheo} and the hypothesis stated above, there is an $\epsilon: SM \to \mathbb{R}_+$ which depends on the curvature tensor $R$ and the real numbers $c$ and $d$ such that if $\| A'(v) \| < \epsilon(v)$ then geodesic flow of the Riemannian manifold $(M,g)$ is partially hyperbolic $dim(E^{\sigma}) = r$, $\sigma = s,u$.
\end{corollary}

\begin{proof}
The proof is straightforward from the theorem \ref{t.maintheo}, if one notices that if $A' = 0$ then $\frac{d}{dt} \mathcal{Q}^e = \widetilde{g}(S^e \cdot,\cdot)$ where $$S^e = \begin{bmatrix} -K_A & 0 & 0 & 0 \\ 0 & Id & 0 & 0 \\ 0 & 0 & 0 & -e^2Id + K_B \\ 0 & 0 & -e^2Id + K_B & 0 \end{bmatrix}$$ so for $(\eta,\varsigma) \in C_+$ we have $\frac{d}{dt} \mathcal{Q}^e \geq g(\varsigma_A,\varsigma_A) + \alpha^2 g(\eta_A,\eta_A) - 2(e^2 + \beta^2) g(\eta_B,\varsigma_B)$ and $g(\eta_A,\varsigma_A) \geq e^2g(\eta_B,\eta_B) + g(\varsigma_B,\varsigma_B) \geq 2e g(\eta_B,\varsigma_B)$. Then,
$$ \frac{d}{dt} \mathcal{Q}^e \geq g(\varsigma_A - \alpha \eta_A,\varsigma_A - \alpha \eta_A) + 2 \alpha g(\eta_A,\varsigma_A) - 2(e^2 + \beta^2) g(\eta_B,\varsigma_B)$$  $$\geq g(\varsigma_A - \alpha \eta_A,\varsigma_A - \alpha \eta_A) + (2 \alpha - e - \frac{\beta^2}{e}) g(\eta_A,\varsigma_A) > 0$$ for $(\eta,\varsigma) \in \mathcal{C}_+$, since $2 \alpha - e - \frac{\beta^2}{e}$ when $e \in (\beta,\alpha)$.
\end{proof}

\begin{remark}
So, if we look at the statements in the corollary, we see that at the moment we have to ask for the existence of an interval between the $r$-biggest eingenvalues of $R(v,\cdot)v$ and the other eiganvalues (second hypothesis), and a non-oscillatory hypothesis for this interval, i.e., it has to have a constant $e$ in the interval which does not depend on $v \in SM$. One good question would be if partial hyperbolicity still holds when there is no such constant.
\end{remark}

\section{Examples}

In this section we show some examples. The first example, in subsection \ref{s.negative}, is the Riemannian manifold of negative curvature. In the case of the Riemannian manifold of negative curvature the criteria is the same as the criteria for hyperbolicity of the geodesic flow. In the subsection \ref{s.phsymmetric} the example satisfies the criteria of partial hyperbolicity. It is also an hyperbolic example. In the subsection \ref{s.nphsymmetric} the example does not satisfy the criteria and is not partially hyperbolic. In subsection \ref{s.nonanosov} we show the last example. For the last example the criteria is satisfied out of a small set of vectors in the unit tangent bundle. The last example is non Anosov and partially hyperbolic \cite{CP}. 

\subsection{Negatively curved manifolds}
\label{s.negative}

In the negatively curved case, the theorem is trivial. In this case, $A(v) = (\mathbb{R}v)^{\bot}$, and there is no need of a $\beta$ function. Suppose $K \leq - \alpha^2$, for a positive real number $\alpha$. Since $(P_A)' = 0$ in this case, the criteria trivially holds:


$$ \frac{d}{dt} g(\eta,\varsigma) = g(\varsigma,\varsigma) - R(v,\eta,v,\eta)  \geq g(\varsigma,\varsigma) + \alpha^2g(\eta,\eta) > 0.$$

\noindent for any $(\eta,\varsigma) \in C_+(x,v)$, $(x,v) \in SM$.

\subsection{Locally symmetric manifolds}
\label{s.lsm}

In this section we look at the case of $(M,g)$ compact locally symmetric manifold of noncompact type. In a previous work we had shown that if rank is one then the geodesic flow is partially hyperbolic, if rank is at least two, then it is not. 

\begin{definition}
A simply connected Riemannian manifold is called symmetric if for every $x \in M$ there is an isometry $\sigma_x: M \to M$ such that
$$\sigma_x(x) = x, d\sigma_x(x) = - id_{T_xM}.$$

\noindent The property of being symmetric is equivalent to:
\begin{itemize}
\item $\nabla R \equiv 0$,
\item if $X(t)$, $Y(t)$ and $Z(t)$ are parallel vector fields along $\gamma(t)$, then $R(X(t),Y(t))Z(t)$ is also a parallel vector field along $\gamma(t)$.
\end{itemize}
\end{definition}

Each locally symmetric space $N$ is the quotient of a simply connected symmetric space $M$ and a group $\Gamma$ acting on $M$ discretly, without fixed points, and isometrically, such that $N = M / \Gamma.$

\subsubsection{Locally symmetric manifolds of noncompact type and of rank one}
\label{s.phsymmetric}

Locally symmetric spaces with non constant negative curvature have the following parallel subspaces of $(\mathbb{R} v)^{\bot}$:
\begin{eqnarray}
\label{d.AB} 
A(x,v) & := & \{ w \in T_xM : K(v,w) = -4 a^2 \}, \\ B(x,v) & := & \{ w \in T_xM : K(v,w) = - a^2 \},
\end{eqnarray} 

\noindent where $a \in \mathbb{R}$.

The curvature tensor for locally symmetric manifolds of noncompact type and rank one is $$R(v,\eta)v = - 4 a^2 \eta_A - a^2 \eta_B,$$ where $v \in SM$.

Partial hyperbolicity follows from:

\begin{eqnarray*} \frac{d}{dt}(g(\eta_A,\varsigma_A) - e^2 g(\eta_B,\eta_B) - g(\varsigma_B,\varsigma_B)) & = & g(\varsigma_A,\varsigma_A) - g(R(v,\eta)v,\eta_A) \\ - 2e^2 g(\eta_B,\varsigma_B) + 2 g(R(v,\eta)v,\varsigma_B) & = &
4a^2 g(\eta_A,\eta_A) + g(\varsigma_A,\varsigma_A) \\ - 2e^2g(\eta_B,\varsigma_B) - 2 a^2  g(\eta_B,\varsigma_B) & > & 4a^2 g(\eta_A,\eta_A) + g(\varsigma_A,\varsigma_A) \\ - (e + \frac{a^2}{e}) g(\eta_A,\varsigma_A) & > & 0
\end{eqnarray*}

\noindent if $e \in (a,2a)$, for if $e \in (a,2a)$ then $e + \frac{a^2}{e} < 4a$. 


\subsubsection{Locally symmetric manifolds of noncompact type and of rank at least two}
\label{s.nphsymmetric}

\begin{definition}
Let $\mathfrak{g}$ be the algebra of Killing fields on the symmetric space $M$, $p \in M$. Define
$$\mathfrak{k} := \{ X \in \mathfrak{g} : X(p) = 0 \},$$
$$\mathfrak{p} := \{ X \in \mathfrak{g} : \nabla X (p) = 0 \}.$$
\noindent For these subspaces of $\mathfrak{g}$, $\mathfrak{k} \oplus \mathfrak{p} = \mathfrak{g}$ and $\mathfrak{k} \cap \mathfrak{p} = \{ 0 \}$, and $T_pM$ identifies with $\mathfrak{p}$.
\end{definition}

\begin{definition}
Given $p \in M$, we define the involution $\phi_p(g): G \to G: g \to \sigma_p \circ g \circ \sigma_p$. Then, we obtain $\theta_p : d \phi_p : \mathfrak{g} \to \mathfrak{g}$. Since $\theta_p^2 = id$ and $\theta_p$ preserves the lie brackets, the properties of this subspaces of $\mathfrak{g}$ are:

\begin{itemize}
\item[i.] $\theta_{p|\mathfrak{k}} = id$,
\item[ii.] $\theta_{p|\mathfrak{p}} = - id$,
\item[iii.] $[\mathfrak{k},\mathfrak{k}] \subset \mathfrak{k}$, $[\mathfrak{p},\mathfrak{p}] \subset \mathfrak{k}$, $[\mathfrak{k},\mathfrak{p}] \subset \mathfrak{p}$,
\end{itemize}
\end{definition}




Fix a maximal Abelian subspace $\mathfrak{a} \subset \mathfrak{p}$. Let $\Lambda$ denote the set of roots determined by $\mathfrak{a}$, and $$\mathfrak{g} = \mathfrak{g}_0 + \sum_{\alpha \in \Lambda} \mathfrak{g}_{\alpha}.$$

\noindent $\mathfrak{g}_{\alpha} = \{ w \in \mathfrak{g} : (ad X) w = \alpha(X) w \}$, $\alpha: \mathfrak{a} \to \mathbb{R}$ is a one-form.

Define a corresponding decomposition for each $\alpha \in \Lambda$, $\mathfrak{k}_{\alpha} = (id + \theta) \mathfrak{g}_{\alpha}$ and $\mathfrak{p}_{\alpha} = (id - \theta) \mathfrak{g}_{\alpha}$. Then:

\begin{itemize}
\item[i.] $id + \theta: \mathfrak{g}_{\alpha} \to \mathfrak{k}_{\alpha}$ and $id - \theta: \mathfrak{g}_{\alpha} \to \mathfrak{p}_{\alpha}$ are isomorphisms,
\item[ii.] $\mathfrak{p}_{\alpha} = \mathfrak{p}_{- \alpha}$, $\mathfrak{k}_{\alpha} = \mathfrak{k}_{- \alpha}$, and $\mathfrak{p}_{\alpha} \oplus \mathfrak{k}_{\alpha} = \mathfrak{g}_{\alpha}  \oplus \mathfrak{g}_{- \alpha}$,
\item[iii.] $\mathfrak{p} = \mathfrak{a} + \sum_{\alpha \in \Lambda} \mathfrak{p}_{\alpha}$, $\mathfrak{k} = \mathfrak{k}_0 + \sum_{\alpha \in \Lambda} \mathfrak{k}_{\alpha}$, where $\mathfrak{k}_0 = \mathfrak{g}_0 \cap \mathfrak{k}$.
\end{itemize}

For $X \in \mathfrak{a}$ we have that, along the geodesic $\gamma$ in $M$ with initial conditions $\gamma(0) = p$, $\gamma'(0) = X$, the Jacobi fields are linear combinations of the following Jacobi fields:

$$cosh(|\alpha(X)|t)v_j(t) \textrm{ and } sinh(|\alpha(X)|t)v_j(t).$$


\begin{proposition}
If $(M,g)$ is a locally symmetric manifold of noncompact type and rank bigger than one, then there is no continuous function $$A: SM \to Gr(r,TM): v \to A(v) \subset T_{p(v)}M$$ satisfying the hypothesis of the theorem \ref{t.maintheo}.
\end{proposition}

So, in the case of rank bigger than one, fix $r < dim M$, pick $v \in T_xM$ such that $A(v) = \oplus_{i=1}^r p_{\alpha_i}$, $i = 1, \ldots, r$, $|\alpha_1| > |\alpha_2| > \ldots > |\alpha_r|$, such that if $\beta \neq \alpha_i$, $\forall i = 1,\ldots,r$, then $\beta(v) < \alpha_i(v)$, $\forall i = 1,\ldots,r$.

Now we pick $(x,v')$ such that $\alpha_1(v') = 0$. Then $A(v') = \oplus_{i=1}^r p_{\beta_i}$, for some $\beta_j \in \Lambda$, $j=1,\ldots,r$, $|\beta_1| > |\beta_2| > \ldots > |\beta_{r}|$. Notice that $\alpha_1(v') = 0$ implies $\beta_j \neq \alpha_1$, $\forall j=1,\ldots,r$. There is no way to go from one decomposition to the other continuously, so there is no way to define a continuous $A$ as in the statement of the theorem. 

\begin{remark}
In the case of locally symmetric manifolds of rank bigger than one, all the three hypothesis stated prior to theorem \ref{t.maintheo} are not satisfied. It would be interesting to look for examples which do not satisfy only one of these hypothesis.
\end{remark}




\subsection{Non Anosov example}
\label{s.nonanosov}

Let $(M,g)$ be a Riemannian manifold, $\nabla$ its Levi-Civita connection, $R$ its curvature tensor. Let $(M,g_1=e^{\alpha}g)$ be Riemannian manifold with a metric in the conformal class of $g$, $\nabla^1$ its Levi-Civita connection, $R^1$ its curvature tensor. Then

$$\nabla^1_X Y = \nabla_X Y + \frac{1}{2} g(\nabla \alpha, X) Y + \frac{1}{2} g(\nabla \alpha, Y) X - \frac{1}{2} g(X,Y) \nabla \alpha,$$

\noindent and if $\alpha$ is $C^2$-close to zero, then

$$R^1(X,Z,X,W) \approx R(X,Z,X,W) - \frac{1}{2} g(X,X) g(\nabla_Z \nabla \alpha,W) - \frac{1}{2} g(Z,W) g(\nabla_X \nabla \alpha,X).$$  

If $(M,g)$ is the locally symmetric space with curvature in $[-4a^2,-a^2]$, the idea is to pick a closed geodesic $\gamma$ without self-intersections, take a tubular neighborhood aroud $\gamma$. Define an orthogonal $x_0,x_1,\ldots,x_{n-1}$ coordinate system in the tubular neighborhood such that along $\gamma$, $\gamma' = \partial_{x_0}$, $K(\gamma',\partial_{xi}) = -4a^2$ for $i=1,\ldots,r$ and $K(\gamma',\partial_{x_i})= - a^2$ for $i=r+1,\ldots,n-1$. Then we are able to define an $\alpha: SM \to \mathbb{R}$ such that for $g_1 = e^{\alpha}g$, $K(\gamma',\partial_{xi}) = -4a^2$ for $i=1,\ldots,r$ and $K(\gamma',\partial_{x_i})= 0$ for $i=r+1,\ldots,n-1$. 
Define $A$ as in subsection \ref{s.lsm}. Then, for the same quadractic form of the theorem \ref{t.maintheo}, the criteria holds at a set $\mathcal{PH} \subset SM$. Let $\mathcal{T}:= SM - \mathcal{PH}$. Any orbit which crosses $\mathcal{T}$ stays there as little time as we want - time depends on the size of the tubular neighborhood. So, with a bit more work, we can show that its geodesic flow is partially hyperbolic (see details in \cite{CP}).





\end{document}